\documentclass[12pt]{amsart}
\usepackage{amscd,amssymb}
\usepackage[arrow,matrix,graph,frame,poly,arc]{xy}

\newtheorem{theorem}[subsection]{Theorem}
\newtheorem{lemma}[subsection]{Lemma}
\newtheorem{prop}[subsection]{Proposition}
\newtheorem{corollary}[subsection]{Corollary}

\theoremstyle{definition}
\newtheorem{remark}[subsection]{Remark}

\newtheorem{example}[subsection]{Example}
\newtheorem*{ack}{Acknowledgment}

\numberwithin{equation}{section}
\newcommand{\VV}{{\mathcal V}}
\newcommand{\RR}{{\mathcal R}}
\newcommand{\EE}{{\mathcal E}}
\newcommand{\PP}{{\mathcal P}}

\newcommand{\wL}{\widehat{\mathbb{L}}}

\newcommand{\Z}{\mathbb{Z}}
\newcommand{\Q}{\mathbb{Q}}

\newcommand{\C}{\mathbb{C}}
\newcommand{\F}{\mathbb{F}}
\newcommand{\TT}{\mathbb{T}}
\newcommand{\AB}{\mathbb{A}}
\newcommand{\LL}{\mathbb{L}}

\renewcommand{\k}{\Bbbk}
\newcommand{\m}{{\mathfrak m}}

\newcommand{\G}{\Gamma}

\DeclareMathOperator{\Hom}{Hom}
\DeclareMathOperator{\Aut}{Aut}
\DeclareMathOperator{\rk}{rk}
\DeclareMathOperator{\gr}{gr}

\DeclareMathOperator{\coker}{coker}

\DeclareMathOperator{\id}{id}

\DeclareMathOperator{\ch}{char}
\DeclareMathOperator{\df}{def}
\DeclareMathOperator{\abf}{abf}
\DeclareMathOperator{\ab}{ab}
\DeclareMathOperator{\ann}{ann}

\newcommand{\surj}{\twoheadrightarrow}

\begin{document}

\title[Characteristic varieties of nilpotent groups]{%
Characteristic varieties of nilpotent groups and applications}

\author[A.~D.~Macinic]{Anca Daniela M\u acinic$^*$}
\address{Inst. of Math. Simion Stoilow,
P.O. Box 1-764,
RO-014700 Bucharest, Romania}
\email{Anca.Macinic@imar.ro}

\author[S.~Papadima]{\c Stefan~Papadima$^*$}
\address{Inst. of Math. Simion Stoilow,
P.O. Box 1-764,
RO-014700 Bucharest, Romania}
\email{Stefan.Papadima@imar.ro}

\thanks{$^*$Partially supported by the CEEX Programme of
the Romanian Ministry of Education and Research, contract
2-CEx 06-11-20/2006.}



\begin{abstract}
We compute the characteristic varieties and the Alexander polynomial of
a finitely generated nilpotent group. We show that the first characteristic variety
may be used to detect nilpotence.
We use the Alexander polynomial to deduce that the only torsion-free,
finitely generated nilpotent groups with positive deficiency are $\Z$ and
$\Z^2$, extending a classical result on nilpotent link groups.
\end{abstract}

\maketitle


\section{Introduction}
\label{sec:intro}

Let $M$ be a connected CW--complex with finite $1$--skeleton. Let $\TT_M$ be
the algebraic group $\Hom (\pi_1(M), \C^*)$, the {\em character torus} of $M$.
We denote by $\C_{\rho}$ the rank one complex local system on $M$ corresponding
to a character $\rho\in \TT_M$. The {\em characteristic varieties} $\VV_k^i (M)$
are defined by
\begin{equation}
\label{eq:defchar}
\VV_k^i (M)= \{ \rho\in \TT_M \; \mid\; \dim H^i(M, \C_{\rho})\ge k \}\, ,
\end{equation}
for $i,k>0$. They emerged from Novikov's work \cite{N} on 
Morse theory for circle-valued functions.
Their importance was recognized in various other areas, and their study was vigorously
pursued. See for instance \cite{DPS1, DPS2, DPS3, DPS4}, where Serre's problem
on fundamental groups of smooth complex algebraic varieties is attacked through the prism
of the cohomology jumping loci from \eqref{eq:defchar}. 

Given a finitely generated group $G$, one may replace $M$ by the classifying space $K(G,1)$
in the above definitions (simply changing $M$ to $G$ in the notation). Here is
our first result (also proved by Alaniya 
\cite{A}, by using Lie algebra techniques, but only for torsion-free groups $G$
and characters $\rho$ belonging to the identity component of $\TT_G$). 
See Theorem \ref{thm:polyvanish} and Example \ref{ex:polyz}
for a more general statement.

\begin{theorem}
\label{thm:vintro}
Let $G$ be a finitely generated nilpotent group. Then
\begin{equation*}
\VV_k^i (G)=
\begin{cases}
\{ 1 \}\, , & {\rm if} \quad b_i(G)\ge k\, ;\\
\emptyset \, , & {\rm otherwise}\, .
\end{cases}
\end{equation*}
\end{theorem}

It is worth pointing out that computing twisted (co)homology is a very
difficult task, in general. In degree one, the characteristic varieties
$\{ \VV_k^1 (G)\}$ depend only on the metabelian quotient $G/G''$. In
particular, there is a systematic way of producing solvable examples with
pretty complicated characteristic varieties $ \VV_1^1 (G)\subseteq \TT_G$,
by Fox calculus. See \cite{DPS4}.

In degree one, Theorem \ref{thm:vintro} says that
\begin{equation}
\label{eq:v11}
\VV^1_1 (G)\subseteq \{ 1\}\, ,
\end{equation}
for a finitely generated nilpotent group. A notable feature of property
\eqref{eq:v11} is that it can distinguish nilpotence from solvability, for
large classes of groups. 

\begin{theorem}
\label{thm:nilpvsmeta}
Let $G$ be a finitely generated, torsion-free metabelian group, with torsion-free
abelianization. Then $G$ is nilpotent if and only if $\VV^1_1 (G)\subseteq \{ 1\}$.
\end{theorem}

See also Theorem \ref{thm:detect} for a similar result.

There is an infinitesimal analog of characteristic varieties, namely the
so-called {\em resonance varieties} $\{ \RR_k^i (M)\subseteq H^1(M, \C)\}_{i,k>0}$,
defined in terms of the cohomology ring of $M$. In degree one, there is a 
completely contrasting resonance counterpart of Theorem \ref{thm:vintro}:
the resonance varieties $\{ \RR_k^1 (G)\}_{k}$ of a two-step nilpotent, torsion-free
group $G$, can be as complicated as those of an  arbitrary complex $M$.
See Remark \ref{rem:resvschar}.

A finitely generated group $G$ has {\em Alexander polynomial} 
$\Delta^G\in \Z[t_1^{\pm 1}, \dots, t_n^{\pm 1}]$, $n=b_1(G)$,
well defined up to units of the Laurent polynomial ring, see e.g. \cite{MM}.
As a second application of Theorem \ref{thm:vintro}, we compute
Alexander polynomials of nilpotent groups.

\begin{corollary}
\label{cor:alexnilp}
If $G$ is a finitely generated nilpotent group, $\Delta^G$ is a non-zero
constant, up to units.
\end{corollary}

A group $G$ has {\em positive deficiency} (notation: $\df (G)>0$) if it has a finite
presentation with strictly more generators than relations. It is well-known that 
fundamental groups of link complements in $S^3$ have positive deficiency, see 
e.g. \cite{EN}. Corollary \ref{cor:alexnilp} in  turn provides the main step in
deriving our third application of Theorem \ref{thm:vintro}.

\begin{theorem}
\label{thm:nilposdef}
The only torsion-free, finitely generated nilpotent groups with positive
deficiency are $\Z$ and $\Z^2$. 
\end{theorem}

Taking into account that link groups are also torsion-free (as follows e.g.
from \cite{HS}), one may view the above theorem as an extension of a 
classical result in low-dimensional topology: the only nilpotent link groups
are $\Z$ and $\Z^2$; see for instance \cite{EM}, \cite{D}, and the references therefrom.

Along the way, we extend another result in classical link theory, which says
that the Alexander polynomial of the link group vanishes at $1$, for links with
at least $3$ components; see  Torres \cite{T}. A similar result holds, not only
for link groups but for all groups whose Alexander ideal is {\em almost principal},
in the sense of \cite{DPS4}. See Proposition \ref{prop:delta1}.

\section{Jumping loci and poly-cyclic groups}
\label{sec:jump}

\subsection{Jump loci}
\label{ss21}

Due to the fact that $\dim H^i (M, \C_1)=b_i(M)$, Theorem \ref{thm:vintro}
is a consequence of the following vanishing result. 

\begin{theorem}
\label{thm:vanish}
For any finitely generated nilpotent group $G$, for all $i\ge 0$ and 
$\rho \ne 1$, $H^i(G, \C_{\rho})=0$.
\end{theorem}

\begin{proof}
By duality, the statement amounts to proving the vanishing of twisted homology.
This we do, by induction on the nilpotence class of $G$. If $G$ is abelian, the
claim follows easily for cyclic groups, by a direct computation using the standard
resolution of $\Z$ over $\Z G$, and then in general, by resorting to the K\" unneth
formula. For the induction step, we use the Hochschild--Serre spectral sequence.
See \cite{B}.

Recall from \cite[p. 171]{B} that, given a group extension,
\begin{equation}
\label{eq:ext}
1\to K\to G \stackrel{p}{\longrightarrow}B \to 1\, ,
\end{equation}
and a $G$--module $M$, there is a spectral sequence,
\begin{equation}
\label{eq:hs}
E^2_{s t}=H_s(B, H_t(K,M))\Rightarrow H_{s+t}(G, M)\, .
\end{equation}
If $P_{\bullet}$ is a $G$--resolution of $\Z$, the $B$--action on 
$H_*(K, M)= H_*(M\otimes_K P)$ is induced by the tensor product of
the $G$--action on $M$ and the $G$--action on $P$.

Assume that $M=\C_{\rho}$, with $1\ne \rho\in \TT_G$, and the $G$--action
on $M$ factors through $B$, that is, $\rho =p^* \rho'$, with 
$1\ne \rho'\in \TT_{B}$. Then $\C_{\rho}$ is $K$--trivial, and
$H_*(K, M)= H_*(K, \Z)\otimes \C$. The $B$--action on $H_*(K, M)$ is then
induced by the tensor product of the $G$--conjugation action on $H_*(K, \Z)$
and the $G$--action on $\C$. If moreover the extension \eqref{eq:ext}
is central, we infer that $H_*(K, \C_{\rho})$ is a direct sum of copies of
$\C_{\rho'}$, over $B$.

Consider now the central extension
\begin{equation}
\label{eq:central}
1\to \G _i G/\G_{i+1} G\rightarrow G/\G_{i+1} G \stackrel{p}{\longrightarrow}
G/\G_{i} G\to 1\, ,
\end{equation}
with $i\ge 2$. Here $\G_ j G$ are the $j$--fold commutators in $G$:
$\G _1G =G$, $\G _2G= G'= (G, G)$, and inductively $\G_ j G = (G, \G_{j-1} G)$,
where $(,)$ stands for the group commutator.

Pick any $1\ne \rho \in \TT_{G/\G_{i+1} G}$. Since $i\ge 2$, 
$\rho =p^* \rho'$, with 
$1\ne \rho'\in \TT_{G/\G_{i} G}$. By the above discussion, 
\[
E^2_{st}= H_s(G/\G_{i} G, \oplus \C_{\rho'})= \oplus H_s(G/\G_{i} G, \C_{\rho'}).
\]
By induction, $E^2_{st}=0$, for all $s,t$. Hence, $H_*( G/\G_{i+1} G, \C_{\rho})=0$,
by \eqref{eq:hs}.

Due to nilpotence, $G/\G_{j} G= G$, for $j$ large.
\end{proof}

Let $M$ be a connected CW complex with finite $1$--skeleton, as before. For
$z\in H^1(M, \C)$, denote by $\mu_z$ left-multiplication by $z$, acting on
$H^{\bullet}(M, \C)$. Since $z^2=0$, $(H^{\bullet}(M, \C), \mu_z)$ is a cochain complex.
The {\em resonance varieties} $\RR^i_k (M)$ are defined by
\begin{equation}
\label{eq:defres}
\RR^i_k (M)= \{ z\in H^1(M, \C)\; \mid \; \dim H^i (H^{\bullet}(M, \C), \mu_z)\ge k \}\, ,
\end{equation}
for $i,k>0$.

Let $\cup_M\colon \wedge^2 H^1(M, \Q)\to H^2(M, \Q)$ be the cup-product. 
Denote by $K_M$ the kernel of $\cup_M$, and by $\mu_M\colon \wedge^2 H^1(M, \Q)\surj DH^2_M$
the corestriction of $\cup_M$ to its image. Plainly, the resonance varieties 
$\RR^1_k (M)$ depend only on $\mu_M$. As before, the same constructions may be done
for a finitely generated group $G$, by taking $M=K(G, 1)$.

The {\em associated graded Lie algebra} of a group $G$, 
$\gr^*(G):= \oplus_{k\ge 1} \G_k G/\G_{k+1} G$, has Lie bracket $[,]$
induced by the group commutator $(,)$. It follows that $\gr^*(G)$ is generated
(as a Lie algebra) by $\gr^1(G)$, and likewise for the rational 
associated graded Lie algebra, $\gr^*(G)\otimes \Q$.

\begin{remark}
\label{rem:dual}
Recall the exact sequence 
\begin{equation}
\label{eq:mu}
0\to K_G\rightarrow \wedge^2 H^1(G, \Q)\stackrel{\mu_G}{\longrightarrow} 
DH^2_G \to 0\, .
\end{equation}
One also has the exact sequence 
\begin{equation}
\label{eq:br}
0\to N_G\rightarrow \wedge^2 \gr^1(G)\otimes \Q \stackrel{\beta_G}{\longrightarrow} 
\gr^2 (G)\otimes \Q \to 0\, ,
\end{equation}
where $\beta_G$ denotes the Lie bracket. It follows from \cite{S} that
\eqref{eq:br} is the vector space dual of \eqref{eq:mu}, for any
finitely generated group $G$.
\end{remark}

\begin{remark}
\label{rem:resvschar}
For any given complex $M$, one may find a two-step nilpotent, torsion-free
group $G$, such that $\RR_k^1 (M) = \RR_k^1 (G)\, , \forall k$. 

Set $\pi =\pi_1(M)$. Since the classifying map $M\to K(\pi, 1)$ induces
over $\Q$ a cohomology isomorphism in degree one and a monomorphism in degree
two, it follows that $\mu_M= \mu_{\pi}$, hence $\RR_*^1 (M) = \RR_*^1 (\pi)$.
Clearly, $\gr^{\le 2}(\pi)\cong \gr^{\le 2}(\pi/ \G_3 \pi)$, as Lie algebras.
In conclusion, the resonance varieties in degree one of a complex depend only on
the third nilpotent quotient of its fundamental group:
$\RR_*^1 (M) = \RR_*^1 (\pi/ \G_3 \pi)$; see Remark \ref{rem:dual}.
Set $N= \pi/ \G_3 \pi$. 
By construction, $N$ is a finitely generated, two-step nilpotent group (i.e., 
$\G _3 N= \{ 1 \}$). 
Finally, we may take $G= N/ {\rm Tors}~(N)$. In this way, 
we may also achieve torsion-freeness,
without changing $\mu$, since $H^*(G, \Q)=H^*(N, \Q)$, as rings; see \cite{HMR}.

Plainly, the groups $\pi$ and $\pi/ \pi''$ have the same third nilpotent quotient.
We infer that in particular resonance in degree one depends only on the 
metabelian quotient of the fundamental group: $\RR_*^1 (M) = \RR_*^1 (\pi/ \pi'')$. 
This is also true for characteristic varieties: $\VV_*^1 (M) = \VV_*^1 (\pi/ \pi'')$,
see e.g. \cite{DPS4}. However, it is no longer true, in general, that 
$\VV^1_1(G)= \VV^1_1(G/ \G_3 G)$. Indeed, when $G$ is a free group on $n\ge 2$
generators, it is easy to see that $\VV^1_1(G)= (\C^*)^n$, while 
$\VV^1_1(G/ \G_3 G)= \{ 1\}$, by Theorem \ref{thm:vintro}. 
\end{remark}

\subsection{Poly-cyclic groups}
\label{ss2p}

Let $\PP$ be a class of groups. Recall that a (length $\ell$) finite normal series 
of a group $G$ is a chain of subgroups,
\begin{equation}
\label{eq:chain}
1= G_{\ell} \subseteq \dots \subseteq G_{i+1}\subseteq G_i\subseteq \dots \subseteq G_1=G\, ,
\end{equation}
with $G_{i+1}$ normal in $G_i$; if all subgroups $G_i$ are normal in $G$, we speak
about an invariant series. The group $G$ is {\em poly}--$\PP$ if it has a
poly--$\PP$ series \eqref{eq:chain}, i.e., $G_i/G_{i+1}\in \PP$, for $1\le i<\ell$.
If moreover the series is invariant, we call $G$ {\em Poly}--$\PP$. In this case,
we have for each $1\le i<\ell$ an extension with kernel in $\PP$,
\begin{equation}
\label{eq:extp}
1\to G_i/G_{i+1}\rightarrow G/G_{i+1}\rightarrow G/G_i\to 1\, .
\end{equation}

It is well-known that the finitely generated nilpotent groups coincide with the
Poly-cyclic groups for which all extensions \eqref{eq:extp} are central. Moreover,
if $G$ is finitely generated nilpotent and all lower central series quotients,
$\G _jG/\G _{j+1}G$, are torsion-free, then $G$ is Poly--$\Z$, with all extensions
central.

\begin{example}
\label{ex:polyz}
Let $A$ be a finitely generated non-trivial abelian group, and $\alpha\in \Aut (A)$.
We may form the semi-direct product extension,
$1\to A\to A\rtimes_{\alpha} \Z \to \Z\to 1$, where $G:= A\rtimes_{\alpha} \Z$ is
finitely generated. Clearly, $G$ is Poly-cyclic (Poly--$\Z$), if $A$ is cyclic
(respectively $A=\Z$). Denote by $t\in G$ the canonical lift of $1\in \Z$.
By the construction of $G$, it is easily seen that the length $i+1$ commutator
$(t, (t,\dots (t,a)\dots ))$ is equal to $(\alpha -\id)^i(a)$, for any $a\in A$.
This remark may be used to show that, if $\alpha= -\id$ and $A=\Z$ or $A=\Z/k \Z$
(with $k$ odd), then $G$ is not nilpotent.
\end{example}

\begin{lemma}
\label{lem:corfinite}
Assume in \eqref{eq:ext} that all groups are finitely generated. If either the
extension is central, or the group $K$ is finite, then $\VV^i_1 (B)\subseteq \{ 1\}$,
for all $i$, implies that $\VV^i_1 (G)\subseteq \{ 1\}$,
for all $i$.
\end{lemma}

\begin{proof}
When the extension is central, one may use the same argument as in the proof of
Theorem \ref{thm:vanish}. (If $\rho\in \TT_G$ is not trivial on $K$, then the 
$E^2$ page from \eqref{eq:hs} vanishes, since $K$ is abelian.) 

Pick $1\ne \rho\in \TT_G$. Assuming $K$ is finite, we have $H_+(K, \C_{\rho})=0$;
see \cite{B}. If $\rho$ is non-trivial on $K$, we also have $H_0(K, \C_{\rho})=0$,
and we are done. Otherwise, $\rho =p^* \rho'$, with $1\ne \rho'\in \TT_B$. Due to
the finiteness of $K$, $H_+(K, \C_{1})=0$. As explained before, 
$H_0(K, \C_{1})=\C_{\rho'}$, over $B$. The spectral sequence \eqref{eq:hs}
collapses, and $H_*(G, \C_{\rho})= E^{\infty}_{*0}= E^{2}_{*0}=H_*(B, \C_{\rho'})=0$.
\end{proof}

By induction on length, we obtain the following extension of Theorem \ref{thm:vanish}.
Note that the generalization is strict; see Example \ref{ex:polyz}.

\begin{theorem}
\label{thm:polyvanish}
Let $G$ be a Poly-cyclic group with the property that all extensions 
\eqref{eq:extp} with infinite kernel are central. Then $\VV^i_1 (G)\subseteq \{ 1\}$,
for all $i$.
\end{theorem}

\begin{lemma}
\label{lem:kisz}
Assume in \eqref{eq:ext} that $K=\Z$ and $B$ is finitely generated. 
If $\VV^1_1(G)\subseteq \{ 1\}$ and $\VV^i_1(B)\subseteq \{ 1\}$, for all $i$,
then the extension is central.
\end{lemma}

\begin{proof}
The conjugation action of $B=G/K$ on $H_1(K,\C_1)=\C$ is encoded by a character
$\gamma\in \TT_B$. We have to show that $\gamma=1$. 

Pick any $\rho'\in \TT_B$. In the spectral sequence \eqref{eq:hs} associated to
$\rho=p^* \rho'$, we have 
\[
H_{>1}(\Z, \C_1)=0\, ,\quad H_{0}(\Z, \C_1)=\C_{\rho'}\quad {\rm and} \quad 
H_{1}(\Z, \C_1)=\C_{\gamma \rho'}\, ,
\]
over $B$. Suppose $\gamma \ne 1$ and take $\rho'=\gamma^{-1}$. Then the $E^2$-page
is concentrated on $E^2_{s1}= H_s(B, \C_1)$, by our assumption on $\VV^*_1(B)$.
This implies that $H_1(G, \C_{\rho})=E^{\infty}_{01}=E^{2}_{01}=H_0(B, \C_1)=\C$.
Consequently, $1\ne \rho\in \VV^1_1(G)$, contradicting the hypothesis.
\end{proof}

We are now in a position to show that \eqref{eq:v11} is a powerful property,
which enables one to detect nilpotence in the class of Poly--$\Z$ groups.
Compare with Example \ref{ex:polyz}, case $A=\Z$.

\begin{theorem}
\label{thm:detect}
Let $G$ be a Poly--$\Z$ group. Then $G$ is nilpotent 
if and only if $\VV^1_1 (G)\subseteq \{ 1\}$.
\end{theorem}

\begin{proof}
We have to show that $G$ must be nilpotent, if \eqref{eq:v11} holds. We induct on
the length $\ell$ of a poly--$\Z$ invariant series \eqref{eq:chain}.

Note first that $\VV^1_1(G/G_i)\subseteq \VV^1_1(G)$, for all $i$. This follows
easily by inspecting the spectral sequence \eqref{eq:hs} in low degrees, for
the extension $1\to G_i\to G\to G/G_i\to 1$ and a character of $G/G_i$. 

It follows that property \eqref{eq:v11} is inherited by the groups $G/G_i$, for
$1\le i<\ell$. Clearly, they all are  Poly--$\Z$, so induction applies. 
Consider now the extensions \eqref{eq:extp}, with kernel $\Z$. We know that
$\VV^1_1(G/G_{i+1})\subseteq \{ 1\}$. We infer from Theorem \ref{thm:vanish} that
$\VV^*_1(G/G_{i})\subseteq \{ 1\}$, since $G/G_i$ is nilpotent, by induction.
Lemma \ref{lem:kisz} tells us that the extension is central, hence 
$G/G_{i+1}$ is nilpotent, too. This gives the desired nilpotence of $G=G/G_{\ell}$.
\end{proof}

\section{Alexander polynomial and metabelian groups}
\label{sec:almeta}

\subsection{Alexander polynomial}
\label{ss22}

Let $M$ be a connected complex, with finite $1$--skeleton and fundamental group
$G=\pi_1(M, m_0)$. Set $G_{\ab}:= G/G'$ and 
$G_{\abf}:= G_{\ab}/ {\rm Tors}~(G_{\ab})\cong \Z^n$.
Let $p: X\to M$ be the Galois $\Z^n$--cover of $M$ corresponding to
the kernel of the canonical surjection, $G\surj G_{\abf}$. The finitely presented 
$\Z \Z^n$--module $A_G:= H_1(X, p^{-1}(m_0); \Z)$ is called the {\em Alexander module},
and depends only on $G$. Denote by $\EE_1(A_G) \subseteq \Z \Z^n$ the first 
{\em elementary ideal} of the Alexander module, and let 
$\Delta^G:= {\rm g.c.d.}~(\EE_1(A_G))\in \Z \Z^n$ be the Alexander polynomial.

If $G$ is a finitely presented group (e.g., a finitely generated nilpotent group),
everything may be computed in terms of a given finite presentation, as follows; see
Fox \cite{F}. Let $M$ be the $2$--complex associated to the presentation
$G=\langle x_1, \dots, x_m \mid w_1, \dots, w_s \rangle$, with cellular chain complex
$(C_{\bullet}, \partial_{\bullet})$. Denote by 
$(\widetilde{C_{\bullet}}, \widetilde{\partial_{\bullet}})$
the equivariant cellular chain complex of the universal cover $\widetilde{M}$,
\begin{equation}
\label{eq:ctilde}
\widetilde{C_{\bullet}}= \dots 0\to \Z G \otimes C_2 
\stackrel{\widetilde{\partial_2}}{\longrightarrow}
\Z G \otimes C_1 
\stackrel{\widetilde{\partial_1}}{\longrightarrow} \Z G \otimes C_0\, .
\end{equation}
The $m\times s$ matrix of ${\widetilde{\partial_2}}$ may be computed by Fox differential
calculus in the free group on $x_1, \dots, x_m$. It is equal to 
$\big ( \frac{\partial w_i}{\partial x_j} \big )$, modulo the defining relations
of $G$. The $1\times m$ matrix of ${\widetilde{\partial_1}}$ is simply
$(x_j -1)$, modulo defining relations. The Alexander module $A_G$ is presented by
\begin{equation}
\label{eq:foxpres}
A_G = \coker \big ( (\Z \Z^n)^s \stackrel{\AB_G}{\longrightarrow}
(\Z \Z^n)^m \big )\, ,
\end{equation}
where $\AB_G:= \Z \Z^n \otimes_{\Z G}\widetilde{\partial_2}$. Finally, one may also
recover $\partial_{\bullet}$ from \eqref{eq:ctilde}, by taking the reduction of 
$\widetilde{\partial_{\bullet}}$ modulo the augmentation ideal of $G$, $I_G\subseteq \Z G$.

\begin{lemma}
\label{lem:converse}
If $\Delta^G (1)=0$, $b_1(G)\ge 2$.
\end{lemma}

\begin{proof}
By the very construction of $\Delta^G$, our assumption implies that 
$\rk (\AB_G(1))\le s-2$, over $\Q$. The conclusion follows by resorting to
\eqref{eq:ctilde}.
\end{proof}

\subsection{Proof of Corollary \ref{cor:alexnilp}}
\label{ss23}

We have to prove that $\Delta^G \doteq 1$ (equality modulo units in $\C \Z^n$), if
$G$ is finitely generated nilpotent. In other words, we must show that the zero set
$V(\Delta^G)\subseteq \Hom (\Z^n, \C^*)=(\C^*)^n$ is empty.

It follows from \cite[Proposition 2.4]{DPS4}, via our Theorem \ref{thm:vintro}, that 
$V(\Delta^G)\subseteq \{ 1\}$. Were $V(\Delta^G)$ non-empty, we would have 
$V(\Delta^G) =\{ 1\}$, in particular $n\ge 2$, by virtue of Lemma \ref{lem:converse}.

Two cases may arise: either $\Delta^G=0$, or $\Delta^G \not\doteq 0,1$. In the first
situation, $\dim V(\Delta^G) =n$, and in the second $\dim V(\Delta^G) =n-1$. 
In both cases, we arrive at a contradiction.

\begin{example}
\label{ex:alexk}
Any non-zero constant may appear in Corollary \ref{cor:alexnilp}. Indeed, Fox calculus
applied to the standard presentation of $G= \Z \oplus \Z/k\Z$ 
shows that $\Delta^G$ equals $k$.
\end{example}

\subsection{Proof of Theorem \ref{thm:nilpvsmeta}}
\label{ss2m}

If $G$ is metabelian (i.e., if $G''$ is trivial), one has an extension 
$1\to B_{\ab}\to G\to G_{\ab}\to 1$,
where the Alexander invariant $B_{\ab}=G'/G''$ is endowed with the canonical
$\Z G_{\ab}$--module structure coming from conjugation in $G$. According 
to \cite[Proposition I.4.1]{HMR}, $G$ is nilpotent if and only if 
$I\subseteq \sqrt{\ann B_{\ab}}$, where $I$ is the augmentation ideal of $\Z G_{\ab}$.
Since $B_{\ab}$ is torsion-free, by assumption, we are left with proving the
above inclusion with $\C$--coefficients.

We also know that $G_{\ab}=G_{\abf}=\Z^n$. Therefore, the zero sets in $(\C^{*})^n$ of the
elementary ideals $\EE_1(A_G)$ and $\EE_0(B_{\ab})$ coincide, away from $1$;
see e.g. \cite[Corollary 2.3]{DPS4}. It is also well-known that in this case
the zero set of $\EE_1(A_G)$ coincides with $\VV^1_1(G)$, away from $1$
(see for instance \cite[Proposition 2.4]{DPS4}). Putting things together, we infer from
\eqref{eq:v11} that $V(\EE_0(B_{\ab}))\subseteq \{ 1\}= V(I)$, since $I$ is generated by
$(t_i-1)_{1\le i\le n}$.

A standard result in commutative algebra \cite[pp. 511--513]{E} says that the zero set
$V(\EE_0(B_{\ab}))$ equals $V(\ann B_{\ab})$. Thus, we have the inclusion 
$V(\ann B_{\ab})\subseteq V(I)$. By Hilbert's Nullstellensatz, 
$I\subseteq \sqrt{\ann B_{\ab}}$.

The proof of Theorem \ref{thm:nilpvsmeta} is complete.

\begin{remark}
\label{rem:exmeta}
Note that torsion-freeness is really necessary  in the above theorem. Indeed,
recall the group $G=\Z /k\Z \rtimes_{- \id}\Z$ from Example \ref{ex:polyz}, which
is clearly finitely generated metabelian, but not nilpotent. However,
$\VV^1_1(G)\subseteq \{ 1\}$ (use Lemma \ref{lem:corfinite}), the reason being the
existence of torsion in $G$.

Having Theorem \ref{thm:detect} in mind, we ought to point out that, in general,
the groups to which the nilpotence test from Theorem \ref{thm:nilpvsmeta} applies
need not be Poly--$\Z$. Indeed, it is easy to check by induction that the commutator
subgroup $G'$ must be finitely generated, if $G$ is Poly--$\Z$. On the other hand,
it is equally easy to see that the finitely generated metabelian quotient $G=\F/\F''$,
where $\F$ is a free group on $n\ge 2$ generators, is torsion-free with torsion-free
abelianization, but $G'$ is not finitely generated.

Finally, note also that there are many finitely generated nilpotent groups with
torsion-free lower central series quotients (hence, Poly--$\Z$), but not metabelian.
\end{remark}

\subsection{A converse to Lemma \ref{lem:converse}}
\label{ss24}

Let $G=\langle x_1, \dots, x_m \mid w_1, \dots, w_s \rangle$ be a finitely presented group.
Set $n= b_1(G)$. We are going to view $\Delta^G$ in $\k [[x_1, \dots, x_n]]$, $\k$ 
a field, by Magnus expansion, $e: \Z \Z^n\to \k [[x_1, \dots, x_n]]$. Here $e$ is the
ring homomorphism defined by sending each $t_i$ to $1+x_i$; it clearly sends 
$I:= I_{\Z^n}= (t_i-1)_i$ into the maximal ideal $\m = (x_i)_i$. Set
$n_p:= b_1(G, \k)$, noting that it depends only on $p=\ch (\k)$, and that $n_0=n$.

We first show how to find a {\em minimal} presentation of 
$\k [[x_1, \dots, x_n]]\otimes_{\Z \Z^n} A_G$.

\begin{lemma}
\label{lem:minalex}
Let $G=\langle x_1, \dots, x_m \mid w_1, \dots, w_s \rangle$ be a finitely presented group,
and $\k$ be a characteristic $p$ field. Then
\[
\widehat{A_G}:= \k [[x_1, \dots, x_n]]\otimes_{\Z \Z^n} A_G =
\coker \big ( \k [[x_1, \dots, x_n]]^t \stackrel{\overline{\AB}}{\longrightarrow}
\k [[x_1, \dots, x_n]]^{n_p} \big )\, ,
\] 
over $\k [[x_1, \dots, x_n]]$, where $\overline{\AB}\equiv 0$ mod $\m$, and $n_p -t=m-s$.
\end{lemma}

\begin{proof}
We know from \eqref{eq:foxpres} that 
$\widehat{A_G}= \coker (\widehat{\AB_G}:= \k [[x_1, \dots, x_n]]\otimes_{\Z \Z^n} \AB_G)$,
over $\k [[x]]:= \k [[x_1, \dots, x_n]]$. Use linear algebra to find $\k$--vector
space decompositions, $C_2\otimes \k =Z_2\oplus B_1$ and $C_1\otimes \k =N_1\oplus B_1$,
with respect to which $\widehat{\AB_G}(1)= 0\oplus \id$. It follows that 
$\dim N_1=n_p$ and $\dim N_1- \dim Z_2=m-s$.

Consider now the $\k [[x]]$--submodule $\widehat{\AB_G}(\k [[x]]\otimes B_1)$
of the free module $\k [[x]]\otimes C_1$. Since the $\m$--adic filtration of 
$\k [[x]]$ is complete, one may look at the associated graded picture to infer from
$\widehat{\AB_G}(1)= 0\oplus \id$ that the above submodule is a free summand, with
complement isomorphic to $\k [[x]]\otimes N_1$. The lemma follows.
\end{proof}

\begin{corollary}
\label{cor:elem}
Let $G$ be a finitely presented group and $\k$ a characteristic $p$ field. Set
$n=b_1(G, \Q)$ and $n_p=b_1(G, \k)$. Then
\[
\EE_i (A_G)\subseteq \m^{n_p-i}\, , \quad {\rm if}\quad i<n_p\, ,
\]
where $\EE_i$ denotes the $i$--th elementary ideal, and $\m$ is the maximal ideal
of the formal power series ring $\k [[x_1, \dots, x_n]]$.
\end{corollary}

\begin{proof}
Due to the natural behaviour of elementary ideals under base change, we may use
Lemma \ref{lem:minalex} to replace $A_G$ by $\coker (\overline{\AB})$. By adding
trivial relations if necessary, we may also assume that $m\le s$. If $i<n_p$,
$\EE_i (\widehat{A_G})$ is generated by the $(n_p-i)$--minors of the minimal presentation
matrix $\overline{\AB}$, and we are done.
\end{proof}

Following \cite{DPS4}, we will say that the {\em Alexander ideal} $\EE_1(A_G)$ is
{\em almost principal}, over a field $\k$, if 
\begin{equation}
\label{eq:aprinc}
I^d \cdot (\Delta^G)\subseteq \EE_1(A_G)\, ,\quad {\rm in}\quad \k G_{\abf}\, ,
\end{equation}
for some $d\ge 0$, where $I$ denotes the augmentation ideal of $G_{\abf}$.

\begin{example}
\label{ex:d}
Property \eqref{eq:aprinc} holds for the following classes of groups:
positive-deficiency groups $G$ with $b_1(G)\ge 2$, for which $d=1$ (see \cite{EN});
fundamental groups of closed, connected, orientable $3$--manifolds, where
$d=2$ (see \cite{MM}).
\end{example}

\begin{prop}
\label{prop:delta1}
Let $G$ be a finitely presented group with $n:= b_1(G, \Q)>0$, and $\k$ a
characteristic $p$ field. Set $n_p:= b_1(G, \k)$. If $\EE_1 (A_G)$ is 
almost principal over $\k$ and $n_p>d+1$, then $\Delta^G\in \m^{n_p -d-1}$,
where $\m$ is the maximal ideal of $\k [[x_1, \dots, x_n]]$. In particular,
$\Delta^G(1)=0\in \k$. 
\end{prop}

\begin{proof}
Corollary \ref{cor:elem} guarantees that $\EE_1 (A_G)\subseteq \m^{n_p-1}$.
Hence, $x_1^d \cdot \Delta^G \in \m^{n_p-1}$, by \eqref{eq:aprinc}. Clearly, we 
may suppose that $\Delta^G \ne 0$ in $\k [[x]]$, and $d>0$. Take the initial term of 
$\Delta^G$ in $\k [[x]]$ to conclude that $\Delta^G\in \m^{n_p -d-1}$, as asserted.
\end{proof}

\begin{remark}
\label{rem:step1}
By resorting to Proposition \ref{prop:delta1} (over $\C$) and Corollary
\ref{cor:alexnilp}, the proof of Theorem \ref{thm:nilposdef} is reduced to
verifying the following claim.
\begin{quotation}
A torsion-free, finitely generated nilpotent group $G$, with $b_1(G)\le 2$ and
$\df (G)>0$, is isomorphic to $\Z$ or $\Z^2$. 
\end{quotation}

\noindent See Example \ref{ex:d}.
\end{remark}

\section{Malcev Lie algebras}
\label{sec:mal}

We will handle the case $b_1(G)\le 2$ by Malcev Lie algebra techniques.

\subsection{Associated graded and Malcev Lie algebras of groups}
\label{ss31}

A {\em Malcev Lie algebra} is a rational Lie algebra $L$, together with a descending, 
complete $\Q$--vector space filtration $\{ F_r L\}_{r\ge 1}$, such that:
$F_1L= L$; $[F_r L,F_s L]\subseteq F_{r+s} L$, for all $r,s$; the 
associated graded Lie algebra, $\gr^*(L):= \oplus_{r\ge 1} F_r L/F_{r+1} L$, is
generated in degree $1$. 
There is a natural Malcev Lie algebra $E_G$, associated to a group $G$. See 
\cite[Appendix A]{Q} for details.

For example, if $G$ is the free group on $x_1,\dots, x_n$, $E_G$ is the free
Malcev Lie algebra $\wL (x_1, \dots, x_n)$, that is, the degree completion of the
free $\Q$--Lie algebra $\LL^*(x_1, \dots, x_n)$ graded by bracket length, endowed 
with the canonical filtration of formal series. If 
$G= \langle x_1, \dots, x_m \mid w_1, \dots, w_s \rangle$,
\begin{equation}
\label{eq:malfp}
E_G \cong \wL (x_1, \dots, x_m)/ \langle \langle r_1, \dots, r_s \rangle \rangle\, ,
\end{equation}
where $\langle\langle \bullet \rangle\rangle$ denotes the closed Lie ideal
generated by $\bullet$. The filtration of $E_G$ comes from $\wL$, and the Lie 
relators $r_i$ are constructed from the corresponding group relators $w_i$, by
Campbell-Hausdorff expansion. See \cite{P}.

A fundamental property of Quillen's construction $E_G$ is the existence of a natural
graded Lie algebra isomorphism,
\begin{equation}
\label{eq:gr}
\gr^*(E_G) \cong \gr^*(G)\otimes \Q\, .
\end{equation}

In dual form, the Malcev Lie algebra of a finitely generated group $G$ is nothing 
else but D. Sullivan's $1$--minimal model of $G$. Consequently

\begin{theorem}[\cite{S}] 
\label{thm:mal1min}
If $\varphi\colon G\to K$ is a morphism between finitely
generated groups, inducing over $\Q$ a cohomology isomorphism in degree $1$ and a 
cohomology monomorphism in degree $2$, then $E_G\cong E_K$, as filtered Lie algebras.
\end{theorem}

\subsection{A second reduction}
\label{ss32}

The lemma below will be used to reduce the claim from Remark \ref{rem:step1} 
to a statement about Malcev Lie algebras.

\begin{lemma}
\label{lem:malab}
A finitely generated nilpotent, torsion-free group $G$ is isomorphic to $\Z^n$ if 
and only if $E_G\cong E_{\Z^n}$, as filtered Lie algebras.
\end{lemma}

\begin{proof}
Assuming $E_G\cong E_{\Z^n}$, we infer from \eqref{eq:gr} that $\gr^1(G)\otimes \Q=\Q^n$
and $\gr^{>1}(G)\otimes \Q=0$. Therefore the groups $\G _i G/\G _{i+1}G$ are finite,
for $i\ge 2$. Since $\G _i G= 1$ for $i>>0$, we deduce that $\G _2 G$ is finite,
hence trivial, by torsion-freeness. This means that $G$ must be abelian, whence
the result.
\end{proof}

\begin{remark}
\label{rem:step2}
Using the above lemma, we may suppose in Remark \ref{rem:step1} that $b_1(G)=2$,
having to show that $\mu_G \ne 0$, in the notation from \S\S \ref{ss21}.

Indeed, if $b_1(G)=0$ then $E_G\cong E_{\{ 1 \}}$, by Theorem \ref{thm:mal1min},
hence $G=\{ 1\}$. But this cannot happen, since the deficiency 
of the trivial group is zero. If $b_1(G)=1$, the same argument shows that
$G\cong \Z$, as claimed. Finally, assume $b_1(G)=2$ and $\mu_G \ne 0$.
Then Theorem \ref{thm:mal1min}, applied to the canonical morphism 
$\varphi \colon G\to G_{\abf}\cong \Z^2$, gives an isomorphism
$E_G\cong E_{\Z^2}$. Invoking once more Lemma \ref{lem:malab}, we may thus complete
the proof of the claim from Remark \ref{rem:step1}.
\end{remark}

\subsection{Minimal Malcev Lie algebras}
\label{ss33}

We will need the following noncommutative analog of the {\em minimal} 
presentation constructed in Lemma \ref{lem:minalex}.

\begin{prop}
\label{prop:malmin}
Let $E = \wL (x_1, \dots, x_m)/ \langle \langle r_1, \dots, r_s \rangle \rangle$ be a
finitely presented Malcev Lie algebra. There is an isomorphism of filtered Lie algebras,
$$E \cong \wL (x'_1, \dots, x'_n)/ \langle \langle r'_1, \dots, r'_t \rangle \rangle \, ,$$
where $r'_j\in F_2 \wL (x')$, for all $j$, and $n-t=m-s$.
\end{prop}

\begin{proof}
Denote by $X$ the $\Q$--vector space with basis $\{ x_1, \dots, x_m\}$. Let $Y$ be
the $\Q$--vector space with basis $\{ y_1, \dots, y_s\}$. Define a $\Q$--linear map,
$r: Y\to \wL (X)$, by $r(y_j)=r_j$, and set $\overline{r} =\pi \circ r$, where 
$\pi\colon \wL (X)\surj \gr^1 (\wL (X))=X$ is the canonical projection. Choose
vector space decompositions, $Y=Z\oplus B$ and $X=N\oplus B$, 
such that $\overline{r}= 0\oplus \id$. Define a filtered Lie algebra map, 
$f: \wL (X)\to \wL (X)$, on the free generators, by: $f(z)=z$, for $z\in N$, 
and $f(z)=r(z)$, for $z\in B$.

We claim that $f$ is a filtered Lie isomorphism. Indeed, $\gr^1(f)=\id$, by construction.
Hence, $\gr^*(f)$ is onto, since $\gr^*(\wL(X))$ is generated in degree one. 
By a dimension argument, $\gr^*(f)$ is an isomorphism. Our claim follows, by completeness
of the Malcev filtration. 

Interpreting $f$ as a change of free generators, we obtain the desired 
minimal presentation, where $n=\dim N$ and $t=\dim Z$.
\end{proof}

\section{Proof of Theorem \ref{thm:nilposdef}}
\label{sec:pf}

\subsection{}
\label{ss41}

To finish the proof, we are left with verifying the following.

\begin{lemma} 
\label{lem:end}
Let $E = \wL (x_1,  x_2)/ \langle \langle r \rangle \rangle$ be a
one-relator Malcev Lie algebra, with $r\in F_2 \wL (x)$. If 
$\dim_{\Q} \gr^* (E)<\infty$, then $r \not\equiv 0$ mod $F_3 \wL (x)$.
\end{lemma}

Indeed, let us start with a finitely generated nilpotent, positive deficiency
group $G$, and assume $b_1(G)=2$, as in Remark \ref{rem:step2}. By \eqref{eq:malfp},
$$E_G \cong \wL (x'_1, \dots, x'_m)/ \langle \langle r'_1, \dots, r'_{m-1} \rangle \rangle \, .$$
Moreover, $\gr^1(E_G)=\Q^2$, see \eqref{eq:gr}. Resorting to Proposition \ref{prop:malmin},
we find a minimal Malcev Lie presentation,
$E_G \cong  \wL (x_1,  x_2)/ \langle \langle r \rangle \rangle$, with $r\in F_2 \wL (x)$.
Since $G$ is nilpotent, we also know that $\dim_{\Q} \gr^* (E_G)<\infty$,
again by \eqref{eq:gr}. 

Granting Lemma \ref{lem:end}, we infer that $\gr^2(G)\otimes \Q= \gr^2(E_G)=0$, since
$\LL^2(x_1, x_2)$ is one-dimensional, generated by $[x_1, x_2]$. This in turn implies that
$\mu_G$ is an isomorphism, see Remark \ref{rem:dual}. In particular, $\mu_G \ne 0$,
as claimed in Remark \ref{rem:step2}.

\subsection{}
\label{ss42}

We embark now on the proof of Lemma \ref{lem:end}. Clearly, $r\ne 0$, if
$\dim \gr^*(E)< \infty$, so $r= r_d +$ higher terms, where 
$0\ne r_d \in \LL^d(x_1, x_2)$ and $d\ge 2$. We have to show that $d=2$. 

By \cite[\S 3]{MP}, we know that $\gr^*(E)= \LL^*(x_1, x_2)/\,  {\rm ideal}\, (r_d)$,
since $r_d$ is an inert Lie element, in the sense of \cite{HL}. We have thus finally 
reduced the proof of Theorem \ref{thm:nilposdef} to the following assertion.

\begin{lemma}
\label{lem:inert}
Let $L^* =\LL^*(x_1, x_2)/\,  {\rm ideal}\, (r_d)$ be a finite dimensional
graded Lie algebra over $\Q$, where $0\ne r_d \in \LL^d(x_1, x_2)$ and $d\ge 2$.
Then necessarily $d=2$.
\end{lemma}

\begin{proof}
Let  $U L^* =\TT^*(x_1, x_2)/\,  {\rm ideal}\, (r_d)$ be the universal
enveloping algebra (where $\TT^*$ denotes the tensor algebra, graded  by
tensor length), with Hilbert series $U(z)$. By inertia, we have
\begin{equation}
\label{eq:ui}
U(z)^{-1}= 1-2z+z^d\, ,
\end{equation}
see \cite[Th\' eor\` eme 2.4]{HL}.

Set $a_i= \dim_{\Q} L^i$. By the Poincar\' e-Birkhoff-Witt theorem,
\begin{equation}
\label{eq:pbw}
U(z)^{-1}= \prod_{i\ge 1} (1-z^i)^{a_i}\, ,
\end{equation}
where the infinite product is a polynomial, since $\dim_{\Q} L^* < \infty$,
by assumption. 

By comparing \eqref{eq:ui} and \eqref{eq:pbw}, we deduce that the polynomial
$1-2z+z^d$ is divisible by $(1-z)^2$, since clearly $a_1= 2$. This can only
happen for $d=2$.
\end{proof}

The proof of Theorem \ref{thm:nilposdef} is complete.

\begin{ack}
We are grateful to Daniel Matei for raising an inspiring question, on the
possibility of characterizing nilpotence of groups by properties of
characteristic varieties. 
\end{ack}

\bibliographystyle{amsplain}

\begin{thebibliography}{00}

\bibitem{A} L.A.~Alaniya,
{\em Cohomology with local coefficients of certain nilmanifolds},
Russian Math. Surveys (5)\textbf{54} (1999), 1019--1020.

\bibitem{B}  K.S.~Brown,
{\em Cohomology of groups}, Grad. Texts in Math., 
vol.~87, Springer-Verlag, New York-Berlin, 1982.

\bibitem{D} J.F.~Davis,
{\em A two-component link with Alexander polynomial one is concordant to
the Hopf link}, preprint {\tt arxiv:math/0408226}.

\bibitem{DPS1} A.~Dimca, S.~Papadima, A.~I.~Suciu, 
{\em Formality, {A}lexander invariants, and a question 
of Serre}, preprint {\tt arxiv:math/0512480}.


\bibitem{DPS2} A.~Dimca, S.~Papadima, A.~Suciu, 
{\em Quasi-{K}\"{a}hler {B}estvina--{B}rady groups}, 
to appear in J. Alg. Geometry; 
available at {\tt arxiv:math/0603446}.

\bibitem{DPS3} A.~Dimca, S.~Papadima, A.~Suciu, 
{\em Non-finiteness properties of fundamental groups of smooth projective
varieties}, preprint {\tt arxiv:math/0609456}.


\bibitem{DPS4} A.~Dimca, S.~Papadima, A.~I.~Suciu, 
{\em Alexander polynomials: {E}ssential variables and 
multiplicities}, to appear in Internat. Math. Res. Notices;
available at  {\tt arxiv:0706.2499}.

\bibitem{E}  D.~Eisenbud,
{\em Commutative algebra with a view towards 
algebraic geometry}, Grad. Texts in Math., vol.~150, 
Springer-Verlag, New~York, 1995. 

\bibitem{EN} D.~Eisenbud, W.~Neumann, 
{\em Three-dimensional link theory and invariants of plane curve
singularities}, Annals of Math. Studies, vol.~110, Princeton 
University Press, Princeton, NJ, 1985.


\bibitem{EM} B.~Evans, L.~Moser,
{\em Solvable fundamental groups of compact $3$-manifolds},
Trans. Amer. Math. Soc. \textbf{168} (1972), 189--210.


\bibitem{F} R.~H.~Fox,
{\em Free differential calculus. \textup{I}. Derivation in the 
free group ring}, Ann. of Math. \textbf{57} (1953), 547--560.

\bibitem{HL} S.~Halperin, J.-M.~Lemaire,
{\em Suites inertes dans les alg\` ebres de Lie gradu\' ees},
Math. Scand. \textbf{61} (1987), 39--67.

\bibitem{HMR} P.~Hilton, G.~Mislin, J.~Roitberg,
{\em Localization of nilpotent groups and spaces},
North-Holland Math. Studies, vol.~15, North-Holland,
Amsterdam, 1975.

\bibitem{HS} J.~Howie, H.~Short, 
{\em The band-sum problem}, J. London Math. 
Soc. (2) \textbf{31} (1985), no.~3, 571--576. 

\bibitem{MP} M.~Markl, S.~Papadima,
{\em Moduli spaces for fundamental groups and link invariants derived from the
lower central series},
Manuscripta Math. \textbf{81} (1993), 225--242.

\bibitem{MM} C.~T.~McMullen,
{\em The Alexander polynomial of a $3$-manifold and the 
Thurston norm on cohomology}, Ann. Sci. \'{E}cole Norm. Sup. 
\textbf{35} (2002), no. 2, 153--171. 


\bibitem{N} S.~P.~Novikov,
{\em {B}loch homology, critical points of functions, 
and closed $1$-forms},  Dokl. Akad. Nauk. SSSR 
\textbf{287} (1986), 1321--1324. 


\bibitem{P} S.~Papadima, 
{\em Finite determinacy phenomena for finitely presented groups},
in: Proceedings of the 2nd Gauss Symposium. Conference A: 
Mathematics and Theoretical Physics (Munich, 1993), 507--528, 
Sympos. Gaussiana, de Gruyter, Berlin, 1995.  

\bibitem{Q}  D.~Quillen,
{\em Rational homotopy theory}, Ann. of Math. 
\textbf{90} (1969), 205--295.

\bibitem{S}  D.~Sullivan,
{\em Infinitesimal computations in topology},
Inst. Hautes \'{E}tudes Sci. Publ. Math.
\textbf{47} (1977), 269--331.

\bibitem{T} G.~Torres,
{\em On the Alexander polynomial},
Ann. of Math. \textbf{57} (1953), 57--89.


\end{thebibliography}

\end{document}